\documentclass[11pt]{article}
\usepackage[dvips]{graphicx}
\usepackage{amssymb}
\usepackage{amsmath}
\usepackage{amsthm}
\usepackage{latexsym}
\usepackage[english]{babel}

\newtheorem{defi}{Definition}[section]
\newtheorem{teo}[defi]{Theorem}
\newtheorem{pro}[defi]{Proposition}
\newtheorem{oss}[defi]{Remark}
\newtheorem{lem}[defi]{Lemma}

\newtheorem{es}[defi]{Example}

\begin{document}


\title{On the Characteristic Curvature Operator}
\author{Vittorio Martino$^{(1)}$ } \addtocounter{footnote}{1}
\footnotetext{Dipartimento di Matematica, Universit\`a di Bologna,
piazza di Porta S.Donato 5, 40127 Bologna, Italy. E-mail address:
{\tt{martino@dm.unibo.it}}}
\date{}
\maketitle

{\noindent\bf Abstract} We introduce the Characteristic Curvature as the curvature of the trajectories of the hamiltonian vector field with respect to the normal direction to the isoenergetic surfaces and by using the Second Fundamental Form we relate it to the Classical and Levi Mean Curvature. Then we prove existence and uniqueness of viscosity solutions for the related Dirichlet problem and we show the Lipschitz regularity of the solutions under suitable hypotheses. Moreover we prove a non existence result on the balls when the prescribed curvature is a positive constant. At the end we show that neither Strong Comparison Principle nor Hopf Lemma do hold for the Characteristic Curvature Operator.

\section{Introduction}
In this paper we introduce the Characteristic Curvature as the curvature of the trajectories of the hamiltonian vector field with respect to the normal direction to the isoenergetic surfaces. Namely, let $H$ be a smooth (let say $C^2$)  hamiltonian function on  $\mathbb{R}^{n+1} \times \mathbb{R}^{n+1}$ equipped with its standard symplectic structure $J$; then the level set $M$ of $H$ corresponding to some noncritical energy value $E$ is a  smooth hypersurface in $\mathbb{R}^{2n+2} $:
$$M=\{ z \in \mathbb{R}^{2n+2} : H(z)=E \}$$
$M$ is sometime referred as a isoenergetic surface of $H$. The hamiltonian vector field $X^H$ is the vector field tangent to $M$, defined as
$$X^H:=J \nabla H$$
The orbits of $X^H$ are the critical points of the Action functional defined on a suitable space of curves; therefore they represent the trajectories of the motion in the generalized phase space. In particular they are curves on $M$: we will define the characteristic curvature $\mathcal{C}^M$ as the normalized curvature of these curves with respect to the normal to $M$. \\
Later, since $M$ is a real hypersurface in $\mathbb{C}^{n+1}$,  by using the Second Fundamental Form and  the Levi Form we relate $\mathcal{C}^M$ to the Classical Mean Curvature  $\mathcal{H}^M$ and to the Levi Mean Curvature $\mathcal{L}^M$. In fact by direct computation it turns out that
$$(2n+1)\mathcal{H}^M=(2n\mathcal{L}^M+\mathcal{C}^M)$$
We want to note that the characteristic curvature $\mathcal{C}^M$  can be use to obtain characterization properties: in fact following the results on the Levi Mean Curvature obtained by Hounie and Lanconelli in \cite{HL} and \cite{HL1} where they proved Alexandrov type theorems for Reinhardt domain in $\mathbb{C}^{2}$ first and under suitable hypotheses in $\mathbb{C}^{n+1}$ for every $n\geq1$ then, it is proved in \cite{io2} an analogous symmetry result for Reinhardt domain in $\mathbb{C}^{n+1}$ using the characteristic curvature $\mathcal{C}^M$.\\
In the sequel we will write explicitly  the corresponding second order differential operator acting on the defining function $H$, in particular when $M$ is locally seen as  the graph of some real function $u$ defined on some open set $\Omega\subseteq\mathbb{R}^{2n+1}$, we will define the characteristic curvature operator $\mathcal{T}$ acting on $u$.
The second order differential operator $\mathcal{T}$ is a quasilinear (highly) degenerate elliptic operator on $\mathbb{R}^{2n+1}$: in fact the principal part depends on the gradient of $u$ and it has $2n$ distinct eigenvectors corresponding to the eigenvalue zero and only one direction of positivity. We will prove under suitable hypotheses existence and uniqueness of viscosity solutions for the associated Dirichlet Problem with prescribed curvature function $k$:
$$\left\{
\begin{array}{ll}
\mathcal{T}u=k & \mbox{in $\Omega,$} \\
u=\varphi, & \mbox{on $\partial\Omega$,}
\end{array}
\right.$$
We will use the classical tools introduced by Crandall, Ishii, Lions in \cite{il}, \cite{cil}. Then we will prove the Lischitz regularity of the solution by using a Bernstein method to obtain a gradient bound for solutions of the regularized operator and then by a limit process argument. Moreover we prove a non existence result on the balls when the prescribed curvature is a positive constant. Similar results were proved by Slodkowski and Tomassini in \cite{st1} for the Levi equation in the case $n=1$, then by Martino and Montanari in \cite{mm} for the Mean Levi Curvature, and by Slodkowski and Tomassini in \cite{st2} and by Da Lio and Montanari in \cite{DM} for the Levi Monge Amp\`{e}re equation.\\
At the end, by mean of two counter examples, we will show that neither the Strong Comparison Principle nor the Hopf Lemma hold for the operator $\mathcal{T}$. This is substantial difference between the highly degenerate Characteristic operator and the Levi Curvature operators for which Lanconelli and Montanari in \cite{ML} proved the Strong Comparison Principle: indeed the principal part of Levi Curvature operators is degenerate only with respect to one direction and when computed on strictly pseudo-convex functions, the $2n$ vector fields respect to which the operator is strictly elliptic satisfy the rank H\"{o}rmander condition.

\section{The characteristic curvature}
Let us consider $\mathbb{R}^{n+1} \times \mathbb{R}^{n+1}$ with its standard Liouville differential 1-form $\lambda$ and its canonical symplectic 2-form $\omega$:\\
if $z=(x,y)\in \mathbb{R}^{n+1} \times \mathbb{R}^{n+1}$, then
$$\lambda=\frac{1}{2}\sum_{k=1}^{n+1} (y_k dx_k - x_k dy_k), \qquad \omega:= d\lambda= \sum_{k=1}^{n+1} (dy_k \wedge dx_k)$$
It holds:
$$\lambda(V)=\frac{1}{2} g(J z,V), \qquad \forall V \in \mathbb{R}^{2n+2}$$
and
$$\omega(V,J U)=d\lambda(V,J U)=g(V,U), \qquad \forall V,U \in \mathbb{R}^{2n+2}$$
where $g(\cdot,\cdot)$ is the standard scalar product in $\mathbb{R}^{2n+2}$, and
$$J=\left(
\begin{array}{cc}
0 & I_{n+1}\\
-I_{n+1} & 0\\
\end{array}
\right)
$$
is the canonical symplectic matrix in $\mathbb{R}^{2n+2}$.
Let us consider now a dynamical system described by a smooth hamiltonian function
$$H:\mathbb{R}^{n+1} \times \mathbb{R}^{n+1} \longrightarrow \mathbb{R}, \qquad z=(x,y)\longmapsto H(x,y)=H(z)$$
and define the Action functional
$$A(\gamma)=\int_{t_0}^{t_1}\Big( \lambda(\dot{\gamma}(t)) -H(\gamma(t)) \Big)dt, \qquad \gamma:[t_0,t_1]\rightarrow\mathbb{R}^{2n+2}$$
By taking the First Variation of $A$ on a suitable space of curves one obtains that critical points of $A$ satisfy the following first order system (Hamilton)
\begin{equation}\label{hamilton}
\left\{
\begin{array}{l}
\dot{x}_k=\displaystyle\frac{\partial H}{\partial y_k}(x,y)\\
\\
\dot{y}_k=\displaystyle -\frac{\partial H}{\partial x_k}(x,y)\\
\end{array}\right.
\qquad k=1,\ldots,n+1
\end{equation}
and a Least Action Principle states that trajectories of motion (in the generalized phase space $\mathbb{R}^{n+1} \times \mathbb{R}^{n+1}$) are solutions of (\ref{hamilton}). Moreover the conservation of energy principle ensures that if $\gamma$ is a critical point for $A$, then $\gamma(t) \in M, \forall t\in [t_0,t_1]$, where $M$ is the hypersurface in $\mathbb{R}^{2n+2}$ $$M=\{ z \in \mathbb{R}^{2n+2} : H(z)=E \}$$
with $E$ some constant such that $DH(z)\neq0$ for all $z\in M$; we will refer to $M$ as the isoenergetic surface of  $H$ of energy $E$.\\
Now by denoting $\displaystyle\frac{\partial H}{\partial x_k}=f_{x_k}$ and
$\displaystyle\frac{\partial H}{\partial y_k}=f_{y_k}$, if
$$DH(x,y)=(f_{x_1},\ldots, f_{x_{n+1}},f_{y_1}, \ldots f_{y_{n+1}})$$
then the hamiltonian vector field for $H$ is the vector field tangent to $M$
$$X^H_z:=\Big(f_{y_1}(z), \ldots f_{y_{n+1}}(z),-f_{x_1}(z),\ldots, -f_{x_{n+1}}(z)\Big)=J \nabla H(z)$$
where, $\nabla H= (DH)^T$. The Hamilton system (\ref{hamilton}) rewrites as
$$\dot{\gamma}(t)=X^H_{\gamma(t)}$$
We want explicitly remark that the direction given by the hamiltonian vector field only depend on $M$ and $J$: in fact if $\widetilde{H}$ is another hamiltonian function having $M$ as its level surface, then the vector fields $X^H=J \nabla H$ and $X^{\widetilde{H}}=J \nabla \widetilde{H}$ are parallel.\\
We want compute the curvature of the trajectories described by the hamiltonian vector field with respect to the normal direction to $M$. Let us introduce the space of these trajectories. Taking the restriction of $\omega$ on $TM$, one has
$$rank(\omega|_{TM})=2n \qquad \mbox{and} \qquad ker(\omega|_{TM})=1$$
We will call the following one-dimensional subspace of the tangent space the space of the characteristic vector fields:
$$K_z=\{ \xi \in T_zM : \; \omega(v,\xi)=0, \; \forall v \in T_zM \}$$
A smooth curve $\gamma \subseteq M$, such that $\dot{\gamma}\in K_\gamma$ is called a characteristic curve on $M$.
Since
$$\omega(v,X^H)=\omega(v,J \nabla H)=g( v,\nabla H) =0, \quad\forall v \in TM$$
therefore $X^H_z \in K_z, \;\forall z \in M$ and its orbits are characteristic curves on $M$.
\begin{defi}\label{characteristiccurvature}
Let $\varepsilon>0$ and $\gamma:(-\varepsilon,\varepsilon)\rightarrow M$ be any smooth curve such that
$$\gamma(0)=z\in M \qquad \mbox{and} \qquad \dot{\gamma}(0)\in K_{\gamma(0)}$$
We will call the characteristic curvature of $M$ at a point $Z$ the following
$$\mathcal{C}^M_{z}:=\frac{g(\ddot{\gamma}(0),N_{z})}{|\dot{\gamma}(0)|^2}$$
where $N_{z}$ is a unit normal direction to $M$ at $z$.\\
We will say $M$ be  strictly $\mathcal{C}$-convex if $\mathcal{C}^M_{z}>0$, for every $z\in M$.
\end{defi}

\noindent
We can obtain a formula for the characteristic curvature only depending on the characteristic curves. In fact let $\gamma:(t_0,t_1)\rightarrow M$ be a characteristic curve, namely $\dot{\gamma}=X^H=J \nabla H $.
A unit normal direction along $\gamma$ is given by
$$N_{\gamma(t)}=-\frac{\nabla H(\gamma(t))}{|\nabla H(\gamma(t))|}=\frac{J\dot{\gamma}(t)}{|\dot{\gamma}(t)|}$$
Then
$$\mathcal{C}^M_{\gamma(t)}:=\frac{g(\ddot{\gamma}(t),J\dot{\gamma}(t))}{|\dot{\gamma}(t)|^3}$$

\begin{oss}
By the previous formula we can see that the characteristic curvature is a scalar invariant under (rigid) symplectic transformations.
\end{oss}

\begin{es}[characteristic curvature of the spheres]\label{characteristiccurvaturesphere}
Let us consider
$$H(x,y)=\frac{|x|^2 + |y|^2}{2}$$
as hamiltonian function; for any positive constant $E$ the isoenergetic surface of $H$ is a sphere $S_R^{2n+1}$ of radius $R=\sqrt{2E}$. We have
$$\nabla H (x,y)=
\left(
\begin{array}{c}
x\\
y\\
\end{array}
\right), \qquad
J \nabla H (x,y)=
\left(
\begin{array}{c}
y\\
-x\\
\end{array}
\right), \qquad
| \nabla H (x,y)|=R
$$
If $\gamma\subseteq S_R^{2n+1}$ solves the hamiltonian system (\ref{hamilton}) then
$$\dot{\gamma}(t)=
\left(
\begin{array}{c}
\dot{x}(t)\\
\dot{y}(t)\\
\end{array}
\right)=
J \nabla H (\gamma(t))=
\left(
\begin{array}{c}
y(t)\\
-x(t)\\
\end{array}
\right), \qquad |\dot{\gamma}(t)|=R$$
The second derivative is
$$\ddot{\gamma}(t)=
\left(
\begin{array}{c}
\ddot{x}(t)\\
\ddot{y}(t)\\
\end{array}
\right)=
\left(
\begin{array}{c}
\dot{y}(t)\\
-\dot{x}(t)\\
\end{array}
\right)=
J \dot{\gamma}(t)
$$
Therefore
$$\mathcal{C}^{S_R^{2n+1}}_{\gamma(t)}:=\frac{g(\ddot{\gamma}(t),J\dot{\gamma}(t))}{|\dot{\gamma}(t)|^3}=
\frac{g(J\dot{\gamma}(t),J\dot{\gamma}(t))}{|\dot{\gamma}(t)|^3}=\frac{1}{R}$$
\end{es}

\begin{es}[characteristic curvature of cylinder type domains - 1]\label{characteristiccurvaturecylinder1}
Let us consider
$$H(x,y)=\frac{|x|^2 }{2}$$
as hamiltonian function in $\mathbb{R}^{2} \times \mathbb{R}^{2}$; for any positive constant $E$ the isoenergetic surface of $H$ is a cylinder domain of type $C_1=S^1_R \times \mathbb{R}^{2}$ with circles of radius $R=\sqrt{2E}$.
We have
$$\nabla H (x,y)=
\left(
\begin{array}{c}
x\\
0\\
\end{array}
\right), \qquad
J \nabla H (x,y)=
\left(
\begin{array}{c}
0\\
-x\\
\end{array}
\right), \qquad
| \nabla H (x,y)|=R
$$
If $\gamma\subseteq C_1$ solves the hamiltonian system (\ref{hamilton}) then
$$\dot{\gamma}(t)=
\left(
\begin{array}{c}
\dot{x}(t)\\
\dot{y}(t)\\
\end{array}
\right)=
J \nabla H (\gamma(t))=
\left(
\begin{array}{c}
0\\
-x(t)\\
\end{array}
\right), \qquad |\dot{\gamma}(t)|=R$$
The second derivative is
$$\ddot{\gamma}(t)=
\left(
\begin{array}{c}
\ddot{x}(t)\\
\ddot{y}(t)\\
\end{array}
\right)=
\left(
\begin{array}{c}
0\\
-\dot{x}(t)\\
\end{array}
\right)=
$$
Therefore
$$\mathcal{C}^{C_1}_{\gamma(t)}:=\frac{g(\ddot{\gamma}(t),J\dot{\gamma}(t))}{|\dot{\gamma}(t)|^3}=
\frac{g(
\left(
\begin{array}{c}
0\\
-\dot{x}(t)\\
\end{array}
\right),
\left(
\begin{array}{c}
-x(t)\\
0\\
\end{array}
\right))}{|\dot{\gamma}(t)|^3}=0$$
\end{es}
\noindent

\begin{es}[characteristic curvature of cylinder type domains - 2]\label{characteristiccurvaturecylinder1}
Let us consider
$$H(x,y)=\frac{x_1^2 + y_1^2}{2}$$
as hamiltonian function in $\mathbb{R}^{2} \times \mathbb{R}^{2}$; for any positive constant $E$ the isoenergetic surface of $H$ is a cylinder domain of type $C_2=S^1_R \times \mathbb{R}^{2}$ with circles of radius $R=\sqrt{2E}$.
We have
$$\nabla H (x,y)=
\left(
\begin{array}{c}
x_1\\
0\\
y_1\\
0\\
\end{array}
\right), \qquad
J \nabla H (x,y)=
\left(
\begin{array}{c}
y_1\\
0\\
-x_1\\
0\\
\end{array}
\right), \qquad
| \nabla H (x,y)|=R
$$
If $\gamma\subseteq C_2$ solves the hamiltonian system (\ref{hamilton}) then
$$\dot{\gamma}(t)=
\left(
\begin{array}{c}
\dot{x_1}(t)\\
\dot{x_2}(t)\\
\dot{y_1}(t)\\
\dot{y_2}(t)\\
\end{array}
\right)=
J \nabla H (\gamma(t))=
\left(
\begin{array}{c}
y_1(t)\\
0\\
-x_1(t)\\
0\\
\end{array}
\right), \qquad |\dot{\gamma}(t)|=R$$
The second derivative is
$$\ddot{\gamma}(t)=
\left(
\begin{array}{c}
\ddot{x_1}(t)\\
\ddot{x_2}(t)\\
\ddot{y_1}(t)\\
\ddot{y_2}(t)\\
\end{array}
\right)=
\left(
\begin{array}{c}
\dot{y_1}(t)\\
0\\
-\dot{x_1}(t)\\
0\\
\end{array}
\right)=
J \dot{\gamma}(t)
$$
Therefore
$$\mathcal{C}^{C_2}_{\gamma(t)}:=\frac{g(\ddot{\gamma}(t),J\dot{\gamma}(t))}{|\dot{\gamma}(t)|^3}=
\frac{g(J\dot{\gamma}(t),J\dot{\gamma}(t))}{|\dot{\gamma}(t)|^3}=\frac{1}{R}$$
\end{es}
\noindent
\begin{oss}
By the previous two examples we can see that the two isometric ipersurfaces $C_1$ and $C_2$ in $\mathbb{R}^{2} \times \mathbb{R}^{2}$ have different characteristic curvature: in fact the isometry that exchanges $x_2$ to $y_1$ is a rigid but not symplectic transformation.
\end{oss}

\section{Relation with the Classical and Levi Mean Curvature}
Let us think of $M$ as a smooth real hypersurface in $\mathbb{C}^{n+1}$ by identifying $\mathbb{C}^{n+1}\approx\mathbb{R}^{2n+2}$, with $z=(z_1,\ldots,z_{n+1}), z_k=x+iy\simeq (x_k,y_k)$. A defining function for $M$ is a function $f:\mathbb{C}^{n+1}\rightarrow \mathbb{R}$ such that
$$\Omega=\{z\in\mathbb{C}^{n+1}:f(z)<0\},\quad M=\partial\Omega=\{z\in\mathbb{C}^{n+1}:f(z)=0\}$$
Therefore we can think of $f=H-E$. Denoting by $N=-\frac{\nabla f}{|\nabla f|}$ the (inner, if $M$ is compact) unit normal,
we define the characteristic direction $T\in TM$ as:
\begin{equation}\label{eq:direzionecaratteristica}
T:=-J(N)=\frac{J \nabla f}{|\nabla f|}
\end{equation}
where $J$ is the standard complex structure in $\mathbb{C}^{n+1}$ and in our case it coincides with the canonical symplectic matrix in $\mathbb{R}^{2n+2}$. Therefore the characteristic direction for $M$ is the normalized hamiltonian vector field.
The complex maximal distribution or Levi distribution $HM$ is the largest
subspace in $TM$ invariant under the action of $J$
\begin{equation}\label{eq:distribuzionedilevi}
HM=TM\cap J(TM)
\end{equation}
i.e., a vector field $X \in TM$ belongs to $HM$ if and only if
also $JX \in HM $. Moreover, every element in $TM$ can
be written as a direct sum of an element of $HM$ and one of the
space generated by $T$,
\begin{equation}\label{eq:sommadiretta}
TM=HM\oplus\mathbb{R}T
\end{equation}
where $dim(HM)=2n$ and the sum is $g$-orthogonal:
\begin{equation}\label{eq:ortogonale}
\forall X\in HM\quad g(T,X)=0
\end{equation}
Let us denote by $\nabla$ the Levi-Civita connection in
$\mathbb{C}^{n+1}$: we recall that both $\nabla$ and $g$ are compatible with the complex structure  $J$, i.e.:
\begin{equation}\label{eq:compatibile}
J\nabla=\nabla J,\quad  g(\cdot,\cdot)=g(J(\cdot),J(\cdot))
\end{equation}

\noindent
The second fundamental form $h$ on $M$ is defined as:
\begin{equation}\label{eq:secondaformafondamentale}
h(V,W)= g(\nabla_V W,N),\;\forall V,W\in TM
\end{equation}

\noindent
The Levi form $l$ is the hermitian operator on $HM$ defined in the following way:\\
$\forall X_1,X_2\in HM$, if $Z_1=X_1-iJ(X_1)$ and
$Z_2=X_2-iJ(X_2)$, then
\begin{equation}\label{eq:formadilevi}
l(Z_1,\bar Z_2)= g({\nabla}_{Z_1} \bar Z_2, N)
\end{equation}
We can then compare the Levi form with the second fundamental form by using
the identity (see \cite{bog}, Chap.10, Theorem 2):
\begin{equation}\label{eq:bogges}
\forall X\in HM,\quad l(Z,\bar Z)=h(X,X)+h(J(X),J(X))
\end{equation}
Let now $\{X_1,\ldots,X_{n},Y_1,\ldots,Y_{n}\}$, with $Y_k=J X_k$, be an orthonormal basis of the horizontal space $HM$; then the Second Fundamental Form has the following structure
$$h=\left(
\begin{array}{cccccc}
 h(X_k,X_k) & h(X_k,Y_j)  & h(X_k,T)  \\
 h(Y_j,X_k) & h(Y_j,Y_j)  & h(Y_j,T)  \\
 h(T,X_k)   & h(T,Y_k)    & h(T,T)    \\
\end{array}\right)$$
with $k$ and $j$ running in $1,\ldots,n$.
Moreover, by the very definition of characteristic curvature we have that for every $z\in M$
$$h_z(T,T)=g(\nabla_{T} T, N)=\mathcal{C}^M_z $$
\begin{oss}
By the previous identification we can see that the characteristic curvature depends only on $M$ and on the complex structure $J$, therefore it is a scalar invariant under (rigid) holomorphic transformations.
\end{oss}

\noindent
The classical mean curvature $H^M$ and the Levi mean curvature $L^M$
are respectively:
\begin{equation}\label{eq:curvaturemedie}
\mathcal{H}^M=\displaystyle\frac{1}{2n+1}tra(h),\quad
\mathcal{L}^M=\displaystyle\frac{1}{n}tra(l)
\end{equation}
where $tra$ is the canonical trace operator. Therefore a direct computation
leads to the relation between $\mathcal{H}^M$, $\mathcal{L}^M$ and $\mathcal{C}^M$:
\begin{equation}\label{eq:legame}
(2n+1)\mathcal{H}^M=(2n\mathcal{L}^M+\mathcal{C}^M)
\end{equation}

\section{The operator}
Let $f$ be a smooth defining function for $M$, $f:\mathbb{R}^{n+1}\times\mathbb{R}^{n+1}\rightarrow \mathbb{R}$ such that
$$\Omega=\{z=(x,y)\in\mathbb{R}^{2n+2}:f(z)<0\},\quad M=\partial\Omega=\{z\in\mathbb{R}^{2n+2}:f(z)=0\}$$
with $\nabla f(z)\neq 0$ for all $z \in M$. The hamiltonian vector field related to $f$ is
$$X^f=J \nabla f=
\left(
\begin{array}{cc}
0 & I_{n+1}\\
-I_{n+1} & 0\\
\end{array}
\right)
\left(
\begin{array}{c}
f_x\\
f_y\\
\end{array}
\right)=
\left(
\begin{array}{c}
f_y\\
-f_x\\
\end{array}
\right)$$
A smooth characteristic curve on $M$ then satisfies
$$\dot{\gamma}(t)=X^f_{\gamma(t)}, \qquad |\dot{\gamma}(t)|=|X^f_{\gamma(t)}|=|\nabla f(\gamma(t))|$$
The second derivative is
$$\ddot{\gamma}(t)=\frac{d}{dt}X^f_{\gamma(t)}=\frac{d}{dt}\Big(J \nabla f(\gamma(t))\Big)=$$
$$=J D^2f(\gamma(t))\dot{\gamma}(t)=J D^2f(\gamma(t))J \nabla f(\gamma(t))$$
Therefore for any $z \in M$ we have
$$\mathcal{C}^M_{z}:=\frac{1}{|\nabla f(z)|^3}g(D^2f(z)J \nabla f(z),J \nabla f(z))$$
Let us introduce the following $(2n+2)\times (2n+2)$ symmetric matrix depending on $Df(z)$:
$$A(D f(z))=J \nabla f(z) \otimes J \nabla f(z)=$$
$$=\left(
\begin{array}{c}
f_y\\
-f_x\\
\end{array}
\right)
\begin{array}{cc}
\Big(f_y & -f_x \Big)\\
 &\\
\end{array}
=
\left(
\begin{array}{rr}
 f_y \otimes f_y  &  -f_y \otimes f_x\\
-f_x \otimes f_y  &   f_x \otimes f_x\\
\end{array}
\right)
$$
Then the characteristic operator $\mathcal{T}$ is the differential second order operator acting on $f$ in the following way:
$$\mathcal{T}f (z):=\frac{1}{|D f(z)|^3}tra \Big( A(Df(z)) D^2f(z) \Big)$$
Now we are interested in finding an expression for $\mathcal{T}$ when we locally consider the ipersurface $M$ as the graph of some function $u:\mathbb{R}^{2n+1}\supseteq \Omega\rightarrow \mathbb{R}$ such that $(\xi,u(\xi)) \in M$ for all $\xi \in \Omega$.
Let us call then
$$x=(x_1,\ldots,x_n), \quad y=(y_1,\ldots,y_n), \quad x_{n+1}=t, \quad y_{n+1}=s, \quad \xi=(x,y,t)$$
and take as defining function
$$f(z)=f(x,y,t,s)=u(x,y,t)-s=u(\xi)-s, \qquad |Df|^2=1+|Du|^2$$
By defining the following symmetric matrix depending on $Du$
\begin{equation}\label{A(Du)}
A(Du)=
\left(
\begin{array}{rrr}
 u_y \otimes u_y  &  -u_y \otimes u_x & -u_y \\
-u_x \otimes u_y  &   u_x \otimes u_x &  u_x \\
-u_y              &   u_x             &   1  \\
\end{array}
\right)
\end{equation}
finally we have
$$\mathcal{T}u:=\displaystyle\frac{1}{(1+|Du|^2)^{\frac{3}{2}}}tra \Big( A(Du)\, D^2u \Big)$$

\begin{es}[n=1]\label{n=1}
Let $\Omega \subseteq \mathbb{R}^{3}$ be an open set, and $u: \Omega\rightarrow \mathbb{R}$ a $C^2$ function. Then
$$A(Du)=
\left(
\begin{array}{ccc}
 u^2_y    &  -u_y u_x & -u_y \\
-u_x u_y  &   u^2_x   &  u_x \\
-u_y      &   u_x     &   1  \\
\end{array}
\right)$$
and
$$\mathcal{T}u:=\displaystyle\frac{1}{(1+|Du|^2)^{\frac{3}{2}}}tra \Big( A(Du)\, D^2u \Big)=$$
$$\displaystyle\frac{1}{(1+|Du|^2)^{\frac{3}{2}}}
\Big( u^2_y u_{xx} + u^2_x u_{yy} + u_{tt} -2 u_x u_y u_{xy} +2 u_x u_{yt} -2 u_y u_{xt} \Big)$$
\end{es}
\noindent
The characteristic operator $\mathcal{T}$ is a second order quasilinear (highly) degenerate elliptic operator on $\mathbb{R}^{2n+1}$: in fact by (\ref{A(Du)}) we can see that the following $2n$ independent vector fields
$$\partial_{x_k} + u_{y_k}\partial_t \quad , \qquad
\partial_{y_k} - u_{x_k}\partial_t \quad ,\qquad k=1,\ldots,n$$
are eigenvectors for $A(Du)$ with eigenvalue identically equals to zero; instead the vector field
$$-u_{y_1}\partial_{x_1} -u_{y_n}\partial_{x_n} +u_{x_1}\partial_{y_1} +u_{x_n}\partial_{y_n} +\partial_t $$
is an eigenvector with eigenvalue equals to $(1+|u_x|^2+|u_y|^2)$.\\
We will call for the sake of simplicity
$$ \widetilde{A}(p)=\frac{1}{(1+|p|^2)^{\frac{3}{2}}}A(p),\quad \forall p\in \mathbb{R}^{2n+1}$$
therefore
$$\mathcal{T}u:=\displaystyle\frac{1}{(1+|Du|^2)^{\frac{3}{2}}}tra \Big( A(Du)\, D^2u \Big)=
tra \Big( \widetilde{A}(Du)\, D^2u \Big)$$
Let now
$$\sigma(Du)=
\left(
\begin{array}{c}
-u_y\\
u_x\\
1\\
\end{array}
\right),\qquad
\widetilde{\sigma}(Du)=\frac{1}{(1+|p|^2)^{\frac{3}{4}}}\sigma(Du)$$
then
$$A(Du)=\sigma(Du)\sigma(Du)^T,\qquad \widetilde{A}(Du)=\widetilde{\sigma}(Du)\widetilde{\sigma}(Du)^T$$

\section{Viscosity solutions}
Let $\Omega$ be a bounded open set in $\mathbb{R}^{n}$ with $n=2N+1$ for some $N>0$. Let us consider $$F:\Omega\times\mathbb{R}\times\mathbb{R}^{n}\times S(n)\longrightarrow\mathbb{R}, \qquad
F(x,r,p,\Lambda):=-tra(\widetilde{A}(p)\Lambda)+k(x,r)$$
where $S(n)$ is the set of the symmetric matrices $n\times n$, and $k:\Omega\times\mathbb{R}\rightarrow \mathbb{R}$ is a continuous function. We recall that $J^{2,+}u(x_0)$ is the set of the pairs $(p,X)\in \mathbb{R}^n \times S(N)$ such that
$$u(x)\leq u(x_0)+\langle  p,(x-x_0)\rangle +\frac{1}{2} \langle X (x-x_0),(x-x_0)\rangle+
o(|x-x_0|^2)$$
 as $ x\rightarrow x_0 $. The set ${J}^{2,-} {u}(x_0)$ is analogously defined.

\begin{defi}\label{viscdef1}
We  say that  $u\in USC(\overline\Omega)$ is a viscosity subsolution  of $F=0$ if
$$F(x,u(x),p,\Lambda)\leq 0\quad \forall x\in\Omega\quad\textrm{and}\quad (p,\Lambda)\in J^{2,+}u(x)$$
Viscosity supersolutions are analogously defined with the right change of signs.\\
A viscosity solution of $F=0$ is a function $u$ that is both subsolution and supersolution.
\end{defi}
\begin{oss}\label{viscdefequiv}
We recall that it is equivalent to say that  $u\in USC(\overline\Omega)$ (resp. $v\in LSC(\overline\Omega)$) is a viscosity subsolution (resp. supersolution) of $F=0$ if for all $\phi\in C^2(\overline\Omega)$ the following holds: at each local maximum $x_0$ (resp. local minimum)  of $u-\phi$ ($v-\phi$)
\[
F(x_0,u(x_0),D\phi(x_0),D^2\phi(x_0))\leq 0 $$
$$ ( {\textrm
resp.}\, F(x_0,v(x_0),D\phi(x_0),D^2\phi(x_0))\geq 0 )
\]
\end{oss}
\begin{defi} \label{viscdef2}
A function $u\in USC(\overline\Omega)$ (resp. $v\in LSC(\overline\Omega)$) is said to be a viscosity subsolution
(resp. supersolution) of the Dirichlet problem
$$\left\{\begin{array}{ll} F(x,u,Du,D^2u)=0 & \mbox{in
$\Omega,$} \\ u(x)=\varphi(x), & \mbox{on $\partial\Omega$,}
\end{array} \right.\eqno(DP)$$
where $\varphi\in C(\partial\Omega)$,   if $u$ is a viscosity subsolution (resp. $v$ is a supersolution) of $F=0$ such that $u\leq\varphi$ (resp. $v\geq \varphi$) on $\partial \Omega.$
\end{defi}

\noindent
In the sequel when we talk about sub- and supersolutions of
$(DP),$ we will always mean in a viscosity sense.\\

\noindent
We explicitly remark   that $u\in C^2(\Omega)$ is
a viscosity solution of $F=0$ if and only if  $u$ is a classical solution of
$F=0$.

\noindent
If $k$ is a prescribed continuous function, non negative and strictly increasing with respect to $u$, then $F$ is proper according the definition in \cite{cil} and then the comparison principle for $F$  holds  (see \cite{cil}). Anyway, since we are interested even at case when the characteristic curvature is constant, we would like to have the comparison principle for $F$ even when $k$ is not strictly increasing with respect to $u$, but it doesn't depend on $x$. \\
In order to prove the comparison principle for this case we will adapt the proof given for the strictly monotone case: we need two standard lemmas and we refer the reader to \cite{cil} for the proofs.

\begin{lem}\label{lem1}
Let $\Omega\subseteq\mathbb{R}^{n}$ and $u\in USC(\overline\Omega),\quad
v\in LSC(\overline\Omega)$. Define
$$\displaystyle M_{\varepsilon}=\sup_{\overline{\Omega}\times\overline{\Omega}}
\left(u(x)-v(y)-\frac{|x-y|^{2}}{2\varepsilon}\right)$$
with $\varepsilon>0$. Let us suppose there exist
$(x_{\varepsilon},y_{\varepsilon})\in\overline{\Omega}\times\overline{\Omega}$,
such that:
$$\displaystyle \lim_{\varepsilon\rightarrow 0}\left(M_{\varepsilon}-
(u(x_{\varepsilon})-v(y_{\varepsilon})-\frac{|x_{\varepsilon}-y_{\varepsilon}|^{2}}{2\varepsilon})\right)=0$$
Then we have:
$$\begin{array}{l}
i)\quad \displaystyle
\lim_{\varepsilon\rightarrow 0}\frac{|x_{\varepsilon}-y_{\varepsilon}|^{2}}{\varepsilon}=0\\\\
ii)\quad \displaystyle
\lim_{\varepsilon\rightarrow 0}M_{\varepsilon}=u(\widehat{x})-v(\widehat{x})=\sup_{\overline{\Omega}}(u(x)-v(x))
\end{array}$$
where $\widehat{x}$ is the limit of $x_{\varepsilon}$ (up to subsequences) as
$\varepsilon\rightarrow 0$.\\
\end{lem}

\begin{lem}\label{lem2}
Let $\Sigma_{i}\subseteq\mathbb{R}^{n_{i}}$ be a locally compact set and $u_{i} \in USC(\Sigma_{i})$,
for $i=1,\ldots,k$. Let us define:
$$\Sigma=\Sigma_{1}\times \ldots\times \Sigma_{k}$$
$$w(x)=u_{1}(x_{1})+\ldots+u_{k}(x_{k}), \qquad\textrm{con}\quad x=(x_{1},\ldots,x_{k})\in \Sigma$$
$$n=n_{1}+\ldots+n_{k}$$ and suppose that
$\widehat{x}=(\widehat{x}_{1},\ldots,\widehat{x}_{k})$ is a local maximum for $w(x)-\varphi(x)$, where $\varphi\in C^{2}$ in a neighborhood of $\widehat{x}$. Then, for every $\varepsilon>0$ there exists
$\Lambda_{i}\in S(n_{i})$ such that
$$(D_{x_{i}}\varphi(\widehat{x}),\Lambda_{i})\in \overline{J}^{2,+}_{\Sigma_{i}}u_{i}(\widehat{x_{i}}),
\quad\textrm{for}\quad i=1,\ldots,k$$ and the diagonal blocks matrix $\Lambda_{i}$ satisfies
$$\displaystyle-\left(\frac{1}{\varepsilon}+||\Phi||\right)I_{n}\leq
\left(\begin{array}{ccc} \Lambda_{1} & \ldots & 0\\
\vdots &  \ddots & \vdots\\
0 &\ldots & \Lambda_{k}
\end{array}\right)
\leq \Phi+\varepsilon \Phi^{2}$$ with
$\Phi=D^{2}\varphi(\widehat{x})\in S(n)$ and the norm for
$\Phi$ is:
$$||\Phi||=\sup\{|\lambda|:\lambda\quad\textrm{is an eigenvalue of $\Phi$}\}
=\sup\{|\langle \Phi\xi,\xi \rangle|:|\xi|\leq1\}$$
\end{lem}

\noindent
We have the following result:

\begin{pro} (comparison principle)\label{tct}\\
Let $\Omega\subseteq\mathbb{R}^{n}$, a bounded open set, and let
$k:\Omega\times\mathbb{R}\rightarrow\mathbb{R}$  be a prescribed continuous function, non negative,
not decreasing with respect to $u$ and not depending on $x$.
Then the comparison principle for $F$ holds, namely:
if $\underline{u}$ and $\overline{u}$ are respectively viscosity sub- and supersolution of $F=0$ in
$\Omega$ such that  $\underline{u}(y)\leq \overline{u}(y)$ for all $y\in
\partial\Omega$, then $\underline{u}(x)\leq \overline{u}(x)$
for every $x\in \overline{\Omega}$.
\end{pro}
\begin{proof}
Let us define for $m\in\mathbb{N}$,
$u_{m}(x)=\underline{u}(x)+\displaystyle\frac{1}{m}h(x)$ with
$h(x)=g\Big(\frac{|x|^2}{2}\Big)$ where $g\in C^2$ and $g',g''>0$.
We have $$Dh(x)=g' x, \qquad D^2h(x)=g'' x\otimes x+g' I_{n}$$
and
$$tra(A(Dh)\;D^{2}h)\geq g'\inf_{p\in\mathbb{R}^n}(tra(A(p))=g'>0$$
Moreover we choose $g$ in such a way that $\|h\|_\infty<+\infty$. Our aim
is  to show that
$$\sup_{\Omega}(u_m-\overline{u})\leq \frac{1}{m}\|h\|_\infty$$
Suppose by contradiction that for all $m$
large enough we have
$$M_m=\max_{\overline\Omega}(u_m-\overline{u})>\frac{1}{m} \|h\|_\infty$$
Since  we have $\underline{u}(x)\le \overline{u}(x)$ for all
$x\in\partial\Omega$,  such a maximum is achieved at an interior
point $\tilde x$ (depending on $m$). For all $\varepsilon>0$ let
us consider the auxiliary function
$$
w_\varepsilon(x,y)=u_m(x)- v(y)-\frac{|x-y|^2}{2\varepsilon}
$$
Let $(x_\varepsilon,y_\varepsilon)$ be a maximum of $w_\varepsilon$
in $\overline \Omega\times \overline \Omega.$ By lemma (\ref{lem1}) we get, up to subsequences,
$x_\varepsilon,y_\varepsilon\to\tilde x\in\overline\Omega,$
and
$$
\begin{array}{c}
 \displaystyle\frac{|x_\varepsilon-y_\varepsilon|^2}{\varepsilon}=o(1), \quad \mbox{as $\varepsilon\rightarrow 0$,}\\
 u_m(x_\varepsilon)- \overline{u}(y_\varepsilon)\to u_m(\tilde x)-\overline{u}(\tilde x)= M_m\\
  u_m(x_\varepsilon)\to u_m(\tilde x), \quad \overline{u}(y_\varepsilon)\to \overline{u}(\tilde x)
\end{array}
$$
We may suppose without restriction that $\tilde x \ne 0$. Since  $\tilde x$ is necessarily in $\Omega,$   for $\varepsilon$ small
enough we have $x_{\varepsilon},y_{\varepsilon}\in\Omega$. There exist now by lemma (\ref{lem2})
$X,Y\in {S}^n$ such that, if $p_\varepsilon:= \displaystyle
\frac{(x_\varepsilon-y_\varepsilon)} {\varepsilon}$,   we have
\begin{eqnarray*} & & (p_\varepsilon,X)\in
{J}^{2,+} {u_m}(x_{\varepsilon}),\qquad (p_\varepsilon,Y)\in{J}^{2,-} \overline{u}(y_{\varepsilon}),
\end{eqnarray*}
\begin{equation}\label{sd2}
-\frac{3}{\varepsilon}Id\leq \left(\begin{array}{cc} X & 0 \\ 0
& -Y \end{array}\right) \leq \frac{3
}{\varepsilon}\left(\begin{array}{cc} I& -I \\ -I & I
\end{array}\right)
 \end{equation}
Moreover  $u_m$ is a strictly viscosity subsolution of
$$F(x,u_m - \frac{1}{m}h(x),Du_m -\frac{1}{m}Dh(x), D^2 u_m)=-\frac{g'}{m}\frac{1}{f(Du_m -\frac{1}{m}Dh(x))}$$
where
$$f(p)=\displaystyle(1+|p|^{2})^{\frac{3}{2}}$$
Therefore By denoting $p_\varepsilon^m=p_\varepsilon-\frac{1}{m}Dh(x)$ we have
$$\frac{g'}{m}f(p_\varepsilon^m)\leq f(p_\varepsilon)F(y_{\varepsilon},\overline{u},p_\varepsilon,Y)-
f(p_\varepsilon^m)F(x_{\varepsilon},\underline{u},p_\varepsilon^m,X)=$$
$$=tra(A(p_\varepsilon^m)X)-tra(A(p_\varepsilon)Y)+
f(p_\varepsilon)k(\overline{u}(y_{\varepsilon}))-f(p_\varepsilon^m)k(\underline{u}(x_{\varepsilon}))$$
Then using (\ref{sd2})  we have
$$\frac{g'}{m}\leq tra(\sigma(p_\varepsilon^m)X\sigma(p_\varepsilon^m)^T-
\sigma(p_\varepsilon)Y\sigma(p_\varepsilon)^T)
+f(p_\varepsilon)k(\overline{u}(y_{\varepsilon}))-f(p_\varepsilon^m)k(\underline{u}(x_{\varepsilon}))\leq$$
$$\leq\frac{3}{\varepsilon}\Big(\sigma(p_\varepsilon^m)-\sigma(p_\varepsilon)\Big)
\Big(\sigma(p_\varepsilon^m)-\sigma(p_\varepsilon)\Big)^T+
f(p_\varepsilon)k(\overline{u}(y_{\varepsilon}))-f(p_\varepsilon^m)k(\underline{u}(x_{\varepsilon}))\leq$$
$$\leq \frac{3 L_{\sigma}^2}{\varepsilon m^2} (g')^2|x_\varepsilon|^2+ f(p_\varepsilon)k(\overline{u}(y_{\varepsilon}))-f(p_\varepsilon^m)k(\underline{u}(x_{\varepsilon}))$$
Now we note that
$$f(p_\varepsilon^m)\approx f(p_\varepsilon), \qquad as \quad m \rightarrow\infty$$
and by hypothesis on $k$ and by lemma (\ref{lem1}) as $\varepsilon$ approaches zero we get
$$k(\overline{u}(\widetilde{x})-k(\underline{u}(\widetilde{x}))\leq0$$
Therefore by choosing $m=\varepsilon^{-2}$ and taking the limit as $\varepsilon$ approaches zero we obtain a contradiction.
\end{proof}

\noindent
We are going to give geometric sufficient conditions on $\Omega$ and on the prescribed curvature $k$ in order to ensure the existence of sub- and supersolutions for $(DP)$. Let us define now the cylinder type hypersurface in $\mathbb{R}^{n+1}$:
$$\Omega_c:= \partial \Omega \times \mathbb{R}$$
In the next result we use the following assumptions:\\
\emph{let $\Omega_c$ be a strictly $\mathcal{C}$-convex hypersurface with}
\begin{equation}\label{curvcilik}
\sup_{s\in\mathbb{R}}k(x,s)<\mathcal{C}^{\Omega_c}_{x} \quad \mbox{\emph{for every $x \in \partial \Omega$.}}
\end{equation}
and\\
\emph{let $R$ be the radius of the smallest ball containing $\Omega$; then}
\begin{equation}\label{diametro}
\sup_{\overline{\Omega}\times\mathbb{R}}k \leq \frac{1}{R}
\end{equation}

\noindent
We can  prove now the following result:
\begin{teo}\label{condgeom}
Let $\partial\Omega\in C^2$ and suppose (\ref{curvcilik}) and (\ref{diametro}) hold.
If $k$ is either strictly increasing with respect to $u$ or not decreasing with respect to $u$ but independent of $x$, then there exist a unique viscosity solution for $(DP)$.
\end{teo}
\begin{proof}
Since we have comparison principle for both cases, by the Perron type theorem in (\cite{il}, Proposition II.1), we have that if there exist a subsolution $\underline{u}$ and a supersolution $\overline{u}$ for $(DP)$ such that $\underline{u}=\overline{u}=\varphi$ on $\partial\Omega$, then there exist a unique viscosity solution for $(DP)$.
Therefore we are interested in finding explicit sub- and supersolutions for $(DP)$.\\
Let $\rho\in C^2$ be a defining function for $\Omega$, namely $\rho:\mathbb{R}^n\longrightarrow\mathbb{R}$, such that
$$\Omega=\{x\in\mathbb{R}^{n}:\rho(x)<0\} , \qquad \partial\Omega=\{x\in\mathbb{R}^{n}:\rho(x)=0\}$$
Let $V_0=\{x\in\mathbb{R}^n: -\gamma_0<\rho(x)<0\}$, $\gamma_0>0$
such that for every $0\leq\gamma\leq\gamma_0$ the cylinder $\Omega^\gamma_c$ still satisfies (\ref{curvcilik})
where $\Omega^\gamma=\{x\in\mathbb{R}^{n}:\rho(x)<-\gamma\}$. Let
$\{\varphi_\varepsilon\}_{\varepsilon>0}$ be a sequences of smooth function uniformly convergent to $\varphi$ on
$\partial\Omega$; let finally $\widetilde{\varphi}_\varepsilon$ be a smooth extension of $\varphi_\varepsilon$ on $\Omega$.
Define
$\underline{u}_\varepsilon(x)=\widetilde{\varphi}_\varepsilon(x)+\lambda\rho(x)$
and
$\overline{u}_\varepsilon(x)=\widetilde{\varphi}_\varepsilon(x)-\lambda\rho(x)$,
with $\lambda>0$. It holds
$\underline{u}_\varepsilon=\overline{u}_\varepsilon=\varphi_\varepsilon$
on $\partial\Omega$ and for $\lambda$ large enough we have
$\underline{u}_\varepsilon\leq\overline{u}_\varepsilon$ on $V_0$. Now by (\ref{curvcilik}),
for every $x\in V_0$:
$$\lim_{\lambda\rightarrow\infty}-tra(\widetilde{A}(D\underline{u}_\varepsilon)D^2\underline{u}_\varepsilon)+
k(x,\underline{u}_\varepsilon)=$$
$$=\lim_{\lambda\rightarrow\infty}-tra(\widetilde{A}(D\widetilde{\varphi}_\varepsilon+ \lambda
D\rho))(D^2\widetilde{\varphi}_\varepsilon+ \lambda D^2\rho))+k(x,\underline{u}_\varepsilon)=
-\mathcal{C}^{\Omega^\gamma_c}_{x}+k(x,s)\leq0$$
and
$$\lim_{\lambda\rightarrow\infty}-tra(\widetilde{A}(D\overline{u}_\varepsilon)D^2\overline{u}_\varepsilon)+
k(x,\overline{u}_\varepsilon)=$$
$$=\lim_{\lambda\rightarrow\infty}-tra(\widetilde{A}(D\widetilde{\varphi}_\varepsilon-\lambda
D\rho))(D^2\widetilde{\varphi}_\varepsilon-\lambda D^2\rho))+k(x,\overline{u}_\varepsilon)=\mathcal{C}^{\Omega^\gamma_c}_{x}+k(x,s)\geq0$$
Now let $x_0$ be the center of the smallest ball $B(x_0,R)$ containing $\Omega$ and let us now introduce the function $h(x)=-\sqrt{R^2-|x|^2}$, so that
$$tra(\widetilde{A}(Dh)D^2h)=\frac{1}{R}$$
and define
$$\underline{v}_\varepsilon=\left\{\begin{array}{l}
\underline{u}_\varepsilon(x) \qquad\forall x\in V_0\\
h(x)-M_1\qquad\forall x\in \Omega\setminus V_0
\end{array}\right.$$

$$\overline{v}_\varepsilon=\left\{\begin{array}{l}
\overline{u}_\varepsilon(x) \qquad\forall x\in V_0\\
M_2\qquad\forall x\in \Omega\setminus V_0
\end{array}\right.$$
with
$$M_1\geq\sup_{V_0}(h(x)-\underline{u}_\varepsilon),\qquad
M_2\geq\sup_{V_0}\overline{u}_\varepsilon$$

\noindent
Therefore $\underline{v}_\varepsilon,\overline{v}_\varepsilon$ are respectively sub- and supersolution of $(DP)$ with boundary data $\varphi_\varepsilon$. Then there exists a unique viscosity solution of $(DP)$ with boundary data $\varphi_\varepsilon$.
From comparison principle
$$\sup_\Omega|u_\varepsilon-u_{\varepsilon'}|=
\sup_{\partial\Omega}|u_\varepsilon-u_{\varepsilon'}|=
\sup_{\partial\Omega}|\varphi_\varepsilon-\varphi_{\varepsilon'}|$$
Since viscosity solutions are stable with respect to uniform convergence (see \cite{cil})
then $u_\varepsilon$ uniformly converges to the unique solution of (DP).
\end{proof}

\section{Lipschitz viscosity solutions}
In this section we are interested in looking for a Lipschitz continuous viscosity solution of $(DP)$. We will regularize in an elliptic way our operator in order to obtain a smooth solution $u_\varepsilon$; then we will prove a uniformly gradient estimate for $Du_\varepsilon$ using a Bernstein method and finally we will get our solution by  taking the uniform limit of $u_\varepsilon$. Let us then set for $0<\varepsilon\leq 1$,
$$A^\varepsilon(p):=A(p)+\varepsilon I_n$$
such that $A^\varepsilon$ is strictly positive definite and
$$
F^\varepsilon(x,u,p,\Lambda):=-tra(\widetilde{A}^\varepsilon(p)\Lambda)+k(x,u)
$$
is elliptic. We are going to consider then the following Dirichlet Problem:
$$\left\{\begin{array}{ll} F^\varepsilon(x,u,Du,D^2u)=0 & \mbox{in
$\Omega,$} \\ u(x)=\varphi(x), & \mbox{on $\partial\Omega$,}
\end{array} \right.\eqno(DP_\varepsilon)$$
We prove
\begin{pro} \label{sg}
Let $k\in C^1 (\overline{\Omega}\times\mathbb{R})$ and $\varphi\in C^{2,\alpha} (\partial\Omega)$, $0<\alpha<1$. If
\begin{equation}\label{eq:sg1}
i)\frac{\partial k}{\partial u}\geq0
\end{equation}
\begin{equation}\label{eq:sg2}
ii) k^2-\sum_{k=1}^n\left|\frac{\partial k}{\partial
x_k}\right|\geq0
\end{equation}
then $(DP_\varepsilon)$ admits a solution
$u^\varepsilon\in C^{2,\alpha}(\overline{\Omega})$ such that
\begin{equation}\label{eq:pmg}
\max_{\overline{\Omega}}|Du^\varepsilon|=\max_{\partial\Omega}|Du^\varepsilon|
\end{equation}
\end{pro}
\begin{proof}
The first statement is a consequence of the ellipticity of $F^\varepsilon$ (see \cite{giltru}).
Now let $\widetilde{A}^\varepsilon=\{\widetilde{a}^\varepsilon_{ij}\}$, therefore we can write
\begin{equation}\label{eq:gra}
-\sum_{i,j=1}^n\widetilde{a}^\varepsilon_{ij}(Du^\varepsilon)
\partial_{ij}u^\varepsilon+k(x,u^\varepsilon)=0
\end{equation}
By differentiating (\ref{eq:gra}) with respect to $x_k$, we get:
$$-\sum_{i,j=1}^n (\sum_{l=1}^n\frac{\partial\widetilde{a}^\varepsilon_{ij}}
{\partial_{l}u^\varepsilon}\partial_{lk}u^\varepsilon)\partial_{ij}u^\varepsilon
-\sum_{i,j=1}^n\widetilde{a}^\varepsilon_{ij}\partial_{ijk}u^\varepsilon
+\frac{\partial k}{\partial x_k}+\frac{\partial k}{\partial
u^\varepsilon}\partial_k u^\varepsilon=0$$
Then multiplying by $\partial_{k}u^\varepsilon$ and summing on $k$
\begin{equation}\label{eq:gra1}
-\sum_{i,j,l,k=1}^n \frac{\partial\widetilde{a}^\varepsilon_{ij}}
{\partial_{l}u^\varepsilon}\partial_{kl}u^\varepsilon
\partial_{ij}u^\varepsilon\partial_{k}u^\varepsilon
-\sum_{i,j,k=1}^n\widetilde{a}^\varepsilon_{ij}
\partial_{ijk}u^\varepsilon\partial_{k}u^\varepsilon+
$$
$$
+\sum_{k=1}^n\frac{\partial k}{\partial
x_k}\partial_{k}u^\varepsilon+\frac{\partial k}{\partial
u^\varepsilon}|Du^\varepsilon|^2=0
\end{equation}
Let us set now
$$v^\varepsilon=|Du^\varepsilon|^2=\sum_{k=1}^n\partial_{k}u^\varepsilon$$
$$\partial_{i}v^\varepsilon=2\sum_{k=1}^n
\partial_{k}u^\varepsilon\partial_{ik}u^\varepsilon$$
$$\partial_{ij}v^\varepsilon=2\sum_{k=1}^n
(\partial_{jk}u^\varepsilon\partial_{ik}u^\varepsilon+
\partial_{k}u^\varepsilon\partial_{ijk}u^\varepsilon)$$
By substituting in (\ref{eq:gra1}), we have
\begin{equation}\label{eq:gra2}
-\sum_{i,j,l=1}^n
\frac{1}{2}\frac{\partial\widetilde{a}^\varepsilon_{ij}}
{\partial_{l}u^\varepsilon}\partial_{ij}u^\varepsilon\partial_{l}v^\varepsilon
-\sum_{i,j=1}^n\frac{1}{2}\widetilde{a}^\varepsilon_{ij}
\partial_{ij}v^\varepsilon+
$$
$$
+\sum_{i,j,k=1}^n\widetilde{a}^\varepsilon_{ij}
\partial_{jk}u^\varepsilon\partial_{ik}u^\varepsilon+
\sum_{k=1}^n\frac{\partial k}{\partial
x_k}\partial_{k}u^\varepsilon+ \frac{\partial k}{\partial
u^\varepsilon}v^\varepsilon=0
\end{equation}
Now by Schwarz theorem and by (\ref{eq:gra}), it holds:
$$\sum_{i,j,k=1}^n\widetilde{a}^\varepsilon_{ij}
\partial_{jk}u^\varepsilon\partial_{ik}u^\varepsilon\geq
\frac{(\sum_{i,j=1}^n\widetilde{a}^\varepsilon_{ij}
\partial_{ij}u^\varepsilon)^2}{tr\widetilde{A}^\varepsilon}\geq (1+v^\varepsilon)^{\frac{1}{2}}k^2$$
Therefore by using (\ref{eq:gra2}) and hypothesis (\ref{eq:sg2})
$$\sum_{i,j=1}^n\frac{1}{2}\widetilde{a}^\varepsilon_{ij}
\partial_{ij}v^\varepsilon +\sum_{i,j,l=1}^n
\frac{1}{2}\frac{\partial\widetilde{a}^\varepsilon_{ij}}
{\partial_{l}u^\varepsilon}\partial_{ij}u^\varepsilon\partial_{l}v^\varepsilon
-\frac{\partial k}{\partial u^\varepsilon}v^\varepsilon=$$
$$=
\sum_{i,j,k=1}^n\widetilde{a}^\varepsilon_{ij}
\partial_{jk}u^\varepsilon\partial_{ik}u^\varepsilon
+\sum_{k=1}^n\frac{\partial k}{\partial
x_k}\partial_{k}u^\varepsilon\geq $$
$$\geq (1+v^\varepsilon)^{\frac{1}{2}}k^2-
\sum_{k=1}^n\left|\frac{\partial k}{\partial x_k}
\right|{v^\varepsilon}^\frac{1}{2}\geq {v^\varepsilon}^\frac{1}{2}
\left(k^2-\sum_{k=1}^n\left|\frac{\partial k}{\partial
x_k}\right|\right)\geq0$$
We can apply the classic maximum principle for elliptic operators (see \cite{giltru}) and we obtain
$$\max_{\overline{\Omega}}|v^\varepsilon|=\max_{\partial\Omega}|v^\varepsilon|$$
and then the result (\ref{eq:pmg}).
\end{proof}

\noindent
Now we can write
$$D u^\varepsilon=(Du^\varepsilon)^\tau+(D u^\varepsilon)^\nu$$
where $(Du^\varepsilon)^\tau$ and $(D u^\varepsilon)^\nu$ are respectively the tangential and normal component of
$Du^\varepsilon$ with respect to $\partial\Omega$: by the previous result we need to estimate only the normal component
$$(D u^\varepsilon)^\nu=\langle D u^\varepsilon,\nu \rangle =\displaystyle \frac{\partial
u^\varepsilon}{\partial \nu}$$
where $\nu$ represents the exterior normal to $\partial\Omega$.
In the next result we use a slightly stronger assumption than (\ref{curvcilik}):\\
\emph{let $\Omega_c$ be a strictly $\mathcal{C}$-convex hypersurface such that there exists a defining function for $\Omega$ $\rho \in C^{2,\alpha}$ with $\triangle \rho >0$ on $\partial \Omega$; and}
\begin{equation}\label{curvcilik1}
\sup_{s\in\mathbb{R}}k(x,s)<\mathcal{C}^{\Omega_c}_{x} \quad \mbox{\emph{for every $x \in \partial \Omega$.}}
\end{equation}
\begin{oss}
The hypothesis of having a defining function with $\triangle \rho >0$ is obviously fulfilled if $\partial \Omega$ is strictly convex; it is also satisfied if the cylinder $\Omega_c$ is strictly pseudoconvex as hypersurface in $\mathbb{C}^{n+1}$.
\end{oss}

\noindent
Then we have:
\begin{pro} \label{sdn}
Let $u^\varepsilon\in C^{2,\alpha}(\overline{\Omega})$ be a solution of $(DP_\varepsilon)$.
If (\ref{curvcilik1}) holds then
\begin{equation}\label{eq:pmgr}
\sup_{\partial\Omega}\left|\frac{\partial u^\varepsilon}{\partial
\nu}\right|\leq C_0
\end{equation}
where $C_0$ depends on $|u^\varepsilon|, D\varphi, D^2\varphi$.
\end{pro}
\begin{proof}
Let $\rho:\mathbb{R}^n\longrightarrow\mathbb{R}$, $\rho\in C^{2,\alpha}$ be a defining function for $\Omega$:
$$\Omega=\{x\in\mathbb{R}^{n}:\rho(x)<0\}, \qquad \partial\Omega=\{x\in\mathbb{R}^{n}:\rho(x)=0\}$$
Let $V_0=\{x\in\mathbb{R}^n: -\gamma_0<\rho(x)<0\}$, $\gamma_0>0$
such that for every $0\leq\gamma\leq\gamma_0$ the cylinder $\Omega^\gamma_c$ still satisfies (\ref{curvcilik})
where $\Omega^\gamma=\{x\in\mathbb{R}^{n}:\rho(x)<-\gamma\}$.
Let $\widetilde{\varphi}$ be a smooth extension of $\varphi$ on $\Omega$. Let us define
$$\underline{u}(x)=\widetilde{\varphi}(x)+\lambda\rho(x), \qquad
\overline{u}(x)=\widetilde{\varphi}(x)-\lambda\rho(x)$$
for any $\lambda>0$. We have
$\underline{u}=\overline{u}=\varphi_\varepsilon$ on
$\partial\Omega$ and $\underline{u}\leq
u^\varepsilon\leq\overline{u}$ on $\{\rho=-\gamma_0\}$ for
$$\lambda>\max\{\frac{1}{\gamma_0}(\max_{\overline{\Omega}}\widetilde{\varphi}+\max_{\overline{\Omega}}|u^\varepsilon|),
\frac{1}{\gamma_0}(\min_{\overline{\Omega}}\widetilde{\varphi}-\max_{\overline{\Omega}}|u^\varepsilon|)\}$$
Therefore $\underline{u}\leq u^\varepsilon\leq\overline{u}$ on $\partial V_0$.
Now by (\ref{curvcilik1}) since $\Delta \rho>0$ is strictly positive in a neighborhood of $\partial\Omega$ we have for $\lambda$ large
$$-tr(\widetilde{A}^\varepsilon(D\underline{u})D^2\underline{u})+k(x,\underline{u})=
-tr(\widetilde{A}(D\underline{u})D^2\underline{u})+k(x,\underline{u})-\varepsilon(\Delta
\widetilde{\varphi}+\lambda\Delta \rho)\leq0$$
and
$$-tr(\widetilde{A}^\varepsilon(D\overline{u})D^2\overline{u})+k(x,\overline{u})=
-tr(\widetilde{A}(D\overline{u})D^2\overline{u})+k(x,\overline{u})-\varepsilon(\Delta
\widetilde{\varphi}-\lambda\Delta \rho)\geq0$$
By the comparison principle we obtain $\underline{u}\leq u^\varepsilon\leq\overline{u}$ on $V_0$ and then
$$\frac{\partial \underline{u}}{\partial\nu}\leq\frac{\partial u^\varepsilon}
{\partial\nu}\leq\frac{\partial \overline{u}}{\partial \nu}$$
on $\partial\Omega$.
\end{proof}

\noindent
Next we estimate $u_\varepsilon$ on $\overline{\Omega}$:
\begin{pro} \label{stiu}
Let $u^\varepsilon\in C^{2,\alpha}(\overline{\Omega})$ be a solution of $(DP_\varepsilon)$. If (\ref{diametro}) holds then:
\begin{equation}\label{eq:stimu}
\sup_{\Omega}|u^\varepsilon|\leq\sup_{\partial\Omega}|u^\varepsilon|+C_1
\end{equation}
\end{pro}
\begin{proof}
Let $x_0$ be the center of the smallest ball $B(x_0,R)$ containing $\Omega$ and let $v(x)=\sqrt{R^2-|x|^2}$. By direct computation ($\triangle v \leq 0$ on $\overline{\Omega}$)
$$F^\varepsilon(v)=- tr(\widetilde{A}(Dv)D^2 v)+k(x,v)+\varepsilon \triangle v \leq -\frac{1}{R}+k(x,v)\leq0$$
and
$$F^\varepsilon(-v)=tr(\widetilde{A}(D(-v))D^2(-v))+k(x,-v) - \varepsilon \triangle v\geq0$$
Therefore $F^\varepsilon(v)\leq F^\varepsilon(u^\varepsilon)$ and
$$\sup_\Omega (v-u^\varepsilon)\leq\sup_{\partial\Omega} (v-u^\varepsilon)$$
$$\inf_\Omega (u^\varepsilon-v)\geq\inf_{\partial\Omega} (u^\varepsilon-v)$$
Analogously  $F^\varepsilon(u^\varepsilon)\leq F^\varepsilon(-v)$ and
$$\sup_\Omega (u^\varepsilon+v)\leq\sup_{\partial\Omega} (u^\varepsilon+v)$$
As we have $v\geq0$ on $\overline{\Omega}$, we proved the desiderated estimate.
\end{proof}

\noindent
To summarize, by the stability of viscosity solutions with respect to uniform convergence, we have proved:
\begin{teo}
Let us suppose that the hypotheses of Propositions (\ref{sg}), (\ref{sdn}), (\ref{stiu}) hold. Then $(DP)$ has a Lipschitz continuous viscosity solution.\\
Moreover, if $k$ is either strictly increasing with respect to $u$ or not decreasing with respect to $u$ but independent of $x$ then the solution is unique.
\end{teo}

\noindent
Next we prove a non existence result on balls when the prescribed curvature is a positive constant, following the idea in \cite{BG}.
\begin{pro}
Let $B:=B(x_0,R)\subseteq \mathbb{R}^{n}$ be the ball with center $x_0$ and radius $R$ and let us suppose that $k$ is a positive constant. If $u$ is a Lipschitz continuous viscosity solution of $F=0$ in $\overline{B}$, then necessarily it holds
$$R \leq \frac{1}{k}$$
\end{pro}
\begin{proof}
Let $0\leq r \leq R$ and let us consider the function
$$\phi(x)=M-\sqrt{r^2-|x-x_0|^2}$$
for some constant $M$.
We have that $\phi\in C^2(B(x_0,r))$ and
$$tra(\widetilde{A}(D\phi)D^2\phi)=\frac{1}{r}$$
By the Lipschitz regularity of $u$ on $\overline{B}$ we can choose $M$ such that $u-\phi$ has a maximum at an interior point
$\bar x \in B(x_0,r)$, then by Remark (\ref{viscdefequiv}) we get ($u$ is a viscosity subsolution of $F=0$ as well)
$$F(\bar x,u(\bar x),D\phi(\bar x),D^2\phi(\bar x))\leq 0$$
that means
$$k\leq tra(\widetilde{A}(D\phi(\bar x))D^2\phi(\bar x))=\frac{1}{r}$$
for every $0\leq r \leq R$. This concludes the proof.
\end{proof}


\section{Some counter examples}
We show by easy counter examples that the Strong Comparison Principle and the Hopf Lemma do not hold for the characteristic operator $\mathcal{T}$. Namely, about the Strong Comparison Principle, if $\Omega$ is a bounded open set in $\mathbb{R}^{n}$ with $n=2N+1$ for some $N>0$, and if we suppose that there exist $u$, $v \in C^2$, such that
$$\left
\{\begin{array}{ll}
\mathcal{T}(u)\geq \mathcal{T}(v) & \mbox{in $\Omega$} \\
u\leq v, & \mbox{in $\Omega$}
\end{array}
\right.$$
then the we cannot conclude that either $u<v$ in $\Omega$ or, if there exists some point $x_0 \in \Omega$
such that $u(x_0)= v(x_0)$, then $u\equiv v$ in $\Omega$.
Regarding the Hopf Lemma if we have
$$\left\{
\begin{array}{ll}
\mathcal{T}(u)\geq \mathcal{T}(v) & \mbox{in $\Omega$} \\
u<v , & \mbox{in $\overline{\Omega}\setminus \{p\}, \quad p\in \partial \Omega$ }\\
u(p)=v(p)  & p\in \partial \Omega
\end{array}
\right.$$
then the we cannot conclude that
$$\displaystyle\frac{\partial u}{\partial\nu}(p)<\frac{\partial v}{\partial\nu}(p)$$
where $\nu$ is the inner normal to $\partial \Omega$ at $p$.

\begin{es}[]\label{examplestrongcomprinc}
Let us consider the ball $B:=B(0,R)\subseteq \mathbb{R}^{n}$ and define the two functions $u,v:B\rightarrow \mathbb{R}$
$$u(x)=-\sqrt{R^2-|x_{n}|^2},\qquad v(x)=-\sqrt{R^2-|x|^2}$$
We have
$$\left
\{\begin{array}{ll}
\mathcal{T}(u)= \mathcal{T}(v)=1/R & \mbox{in B} \\
u\leq v, & \mbox{in B}
\end{array}
\right.$$
and $u(x)=v(x)$ for all the $x\in B$ of the form $x=(0,\ldots,0,x_{n})$.
\end{es}
\noindent

\begin{es}[]\label{examplestrongcomprinc}
Let us consider the two functions of the previous example. Let
$$D=\{x\in \mathbb{R}^{n},\; s.t. \quad g(x)<0 \},\qquad g(x)= x_2^2+\ldots+x_{n}^2-x_1$$
and define $\Omega:=B\cap D$. Let $p=(0,\ldots,0)\in \partial \Omega$ then
$$\nu=-Dg(p)=(1,0,\ldots,0),\qquad Du(p)=Dv(p)=0$$
We have
$$\left
\{\begin{array}{ll}
\mathcal{T}(u)\geq \mathcal{T}(v) & \mbox{in $\Omega$} \\
u<v , & \mbox{in $\overline{\Omega}\setminus \{p\}, \quad p\in \partial \Omega$ }\\
u(p)=v(p)  & p\in \partial \Omega
\end{array}
\right.$$
and
$$\displaystyle\frac{\partial u}{\partial\nu}(p)=\frac{\partial v}{\partial\nu}(p)=0$$
\end{es}
\noindent


\addcontentsline{toc}{section}{Riferimenti Bibliografici}


\end{document}